\documentclass{article}

\usepackage{geometry}

\usepackage[english]{babel}
\usepackage{amsmath,amsthm,amssymb}
\usepackage{tikz-cd}
\usepackage{dsfont}
\usepackage{picture,float,graphicx,listings}
\usepackage{hyperref,verbatim}

\usepackage{todonotes}

\newtheorem{theorem}{Theorem}[section]
\newtheorem{lemma}[theorem]{Lemma}
\newtheorem{corollary}[theorem]{Corollary}
\newtheorem{proposition}[theorem]{Proposition}
\theoremstyle{remark}
\newtheorem*{remark}{Remark}

\theoremstyle{definition}
\newtheorem{definition}[theorem]{Definition}
\newtheorem{example}[theorem]{Example}

\newcommand{\N}{\mathbb{N}}
\newcommand{\Z}{\mathbb{Z}}
\newcommand{\Q}{\mathbb{Q}}
\newcommand{\R}{\mathbb{R}}
\newcommand{\C}{\mathbb{C}}

\renewcommand{\H}{\mathbb{H}}

\renewcommand{\P}{\mathbb{P}}

\newcommand{\rk}{\mathrm{rk}}
\newcommand{\Rk}{\mathrm{Rk}}
\newcommand{\Deg}{\mathrm{Deg}}

\newcommand{\Hom}{\mathrm{Hom}}

\newcommand{\tr}{\mathrm{tr}}

\newcommand{\Vol}{\mathrm{Vol}}

\newcommand{\Coh}{\mathrm{Coh}}

\newcommand{\Id}{\mathrm{Id}}

\let\isanspoint\i
\renewcommand{\i}{\boldsymbol{\mathrm{i}}}
\newcommand{\e}{\boldsymbol{\mathrm{e}}}

\newcommand{\abs}[1]{\left| #1 \right|}

\newcommand{\cA}{\mathcal{A}}
\newcommand{\cE}{\mathcal{E}}
\newcommand{\cF}{\mathcal{F}}
\newcommand{\cG}{\mathcal{G}}
\newcommand{\cI}{\mathcal{I}}
\newcommand{\cL}{\mathcal{L}}
\newcommand{\cO}{\mathcal{O}}

\newcommand{\cS}{\mathcal{S}}
\newcommand{\cZ}{\mathcal{Z}}
\newcommand{\dbar}{\overline{\partial}}
\newcommand{\lex}{\leq_{\mathrm{lex}}}
\newcommand{\gex}{\geq_{\mathrm{lex}}}
\newcommand{\ltx}{<_{\mathrm{lex}}}
\newcommand{\gtx}{>_{\mathrm{lex}}}

\newcommand{\ch}{\mathrm{ch}}

\newcommand{\codim}{\mathrm{codim}}
\newcommand{\supp}{\mathrm{supp}}
\newcommand{\td}{\mathrm{td}}
\newcommand{\Td}{\mathrm{Td}}

\renewcommand{\leq}{\leqslant}
\renewcommand{\geq}{\geqslant}

\title{Polynomial Bridgeland Stability Conditions\\ on the Category of Coherent Sheaves}
\author{Rémi Delloque}
\date{}

\begin{document}

\maketitle

\begin{abstract}
    In this short note, we provide a broad class of examples of stability conditions on the category of coherent sheaves which generalise Gieseker stability. We refer to them as \textit{adapted to coherent sheaves} and they admit Harder--Narasimhan and Jordan--Hölder filtrations. We also study the particular case of Bayer's polynomial Bridgeland stability conditions, and their relation to the gauge theoretical counterpart introduced by Dervan--McCarthy--Sektnan.
\end{abstract}

\section{Introduction}

\paragraph{Background.} In \cite{Bayer}, Bayer introduced a generalisation of Bridgeland stability conditions known as \textit{polynomial Bridgeland stability conditions}. Let $(X,[\omega])$ be a smooth projective variety of dimension $n$. Bayer's idea is to consider a central charge $Z$ \textit{i.e.} a groups morphism
$$
Z : K(X) \longrightarrow \C
$$
of the following form
$$
Z(\cE) = \left(\left(\sum_{i = 0}^n \rho_i[\omega]^i\right) \cup \left(\sum_{j = 0}^n U_j\right) \cup \ch(\cE)\right)^{(n,n)}.
$$
Here, the $\rho_i$ are complex numbers satisfying for all $i$, $\Im(\overline{\rho_i}\rho_{i - 1}) > 0$. The $U_j$ are complex cohomology classes of bidegree $(j,j)$, $U_j \in H^{j,j}(X,\C)$ with $U_0 = 1$. The superscript $(n,n)$ means that we only keep the bidegree $(n,n)$ part. We identify $H^{n,n}(X,\C)$ and $\C$ via integration.

Then, Bayer studies the \textit{large volume regime}, \textit{i.e.} what happens when the ample class $[\omega]$ is replaced by $k[\omega]$ for a large integer $k$.
$$
Z_k(\cE) = \left(\left(\sum_{i = 0}^n \rho_ik^i[\omega]^i\right) \cup \left(\sum_{j = 0}^n U_j\right) \cup \ch(\cE)\right)^{(n,n)}.
$$
This defines a groups morphism
$$
K(X) \longrightarrow \C[k].
$$

Then, there is a full abelian sub-category $\cA \subset D^b(X)$ of the derived category of $X$ such that for every object $\cE \in \cA$, $Z_k(\cE)$ lies in the upper half plane when $k$ is large enough.
\begin{definition}[{\cite[Proposition 1.2.1]{Bayer},\cite[Definition 2.7]{DMS}}]\label{DEF:Stabilité Bayer}
    We say that $\cE$ is asymptotically\\ $Z$-(semi-)stable if for all non-zero sub-object $\cF \in \cA$ of $\cE$,
    $$
    \Im(\overline{Z_k(\cE)}Z_k(\cF)) < (\leq)\, 0.
    $$
    Equivalently, the argument of $Z_k(\cE)$ in $]0,\pi[$ is eventually larger than the argument of $Z_k(\cF)$ in $]0,\pi[$.
\end{definition}
Bayer shows that this stability notion is well-behaved in the sense that there are Harder--Narasimhan filtrations for example. It generalises the Bridgeland stability conditions. See \cite[Theorem 3.2.2]{Bayer} for more details about the construction of $\cA$ and the properties of this stability condition. In this note, we shall focus on the case where $\cA = \Coh(X)$ is the category of coherent sheaves on $X$.

On the analytic side, Dervan--McCarthy--Sektnan showed a Kobayashi--Hitchin like correspondence which establishes a link between asymptotic $Z$-stability and the existence of special metrics \cite{DMS}. It generalises a result of Leung about Gieseker stability \cite{Leung}.

Let $\cE = (E,\dbar)$ be a holomorphic vector bundle on $X$. When $h$ is a Hermitian metric on $\cE$, we define its reduced curvature form
$$
\hat{F}(h) = -\frac{1}{2\i\pi}\nabla \circ \nabla,
$$
where $\nabla$ is the Chern connection associated with $\dbar$ and $h$. Consider the following operator associated to the polynomial central charge $Z$.
$$
\cZ_k(\cE,h) = \left(\left(\sum_{i = 0}^n \rho_ik^i\omega^i\right) \wedge \left(\sum_{j = 1}^n u_j\right) \wedge \e^{\hat{F}(h)}\right)^{(n,n)}.
$$
Here, the only restriction on $\rho$ is that $\Im(\overline{\rho_n}\rho_{n - 1}) > 0$ and $\Im(\rho_n) > 0$ (this second hypothesis is simply a normalisation condition and it can be dropped). The classes $U_j$ are assumed to be real (this assumption is natural in this context, though not necessary, see \cite[Remark 2.9]{DMS}). The $u_j$ are real representatives of the classes $U_j$. By Chern--Weil theory, $\tr(\cZ_k(\cE,h))$ is closed and represents $Z_k(\cE)$. Once more, we are interested in the large volume limit \textit{i.e.} when $k \rightarrow +\infty$.

\begin{definition}[{\cite[Definition 2.22]{DMS}}]
    A collection of Hermitian metrics $(h_k)_{k \geq k_0}$ is a solution to the asymptotic $Z$-critical equation if this family is bounded as tensors for the $\mathcal{C}^2$ norm and for all $k$,
    $$
    \Im(\overline{Z_k(\cE)}\cZ_k(\cE,h_k)) = 0.
    $$
\end{definition}

Dervan--McCarthy--Sektnan showed a Kobayashi--Hitchin like correspondence.

\begin{theorem}[{\cite[Theorem 1.1]{DMS}}]\label{THE:DMS}
    Let $\cE$ be a simple holomorphic vector bundle. If $\cE$ is slope semi-stable (with respect to $[\omega]$) and sufficiently smooth (\textit{i.e.} its graded object can be chosen to be locally free), then we have equivalence between
    \begin{enumerate}
        \item $\cE$ admits a solution to the asymptotic $Z$-critical equation.
        \item $\cE$ is not destabilised by its sub-bundles in the sense of Definition \ref{DEF:Stabilité Bayer}.
    \end{enumerate}
\end{theorem}

The assumption that $\cE$ is slope semi-stable is very natural because the leading order term of the polynomial $\Im(\overline{Z_k(\cE)}Z_k(\cF))$ is a positive multiple of the difference between the slope of $\cF$ and the slope of $\cE$. The assumption that $\cE$ is sufficiently smooth is purely technical and is useful for the analysis.

\paragraph{Motivations.} It naturally raises the question of asymptotic stability with respect to sub-sheaves (like for the Kobayashi--Hitchin correspondence). In particular, can a bundle that admits asymptotically $Z$-critical metrics be destabilised by a sub-sheaf with a non locally free quotient ?

Example \ref{EX:Contre-exemple conjecture DMS} below shows that the structure sheaf of a manifold of dimension greater than or equal to $3$ is asymptotically destabilised by some ideal sub-sheaves when we consider the central charge of the deformed Hermitian Yang--Mills equation. Yet, its flat metric satisfies the deformed Hermitian Yang--Mills equation for any $k$. It gives a counter-example to \cite[Conjecture 1.6]{DMS}. One of the main motivations of this paper is to find which additional assumptions are required to avoid this kind of counter-example toward a generalisation of Theorem \ref{THE:DMS}.

We construct very general stability conditions which avoid such pathological behaviour where a sheaf $\cE$ of dimension $d$ is destabilised by a sub-sheaf $\cF$ whose quotient $\cE/\cF$ has a dimension smaller than $d$. We call these stability conditions \textit{adapted to coherent sheaves of dimension $d$}. We show under the assumption of being adapted the existence of a maximal destabilising sub-sheaf hence the existence of a Harder--Narasimhan filtration. In particular, by uniqueness of this filtration, it is expected to be a filtration by equivariant sub-sheaves when we work with an equivariant sheaf on a toric variety. Hopefully, we expect that an equivariant sheaf is stable if and only if it is stable with respect to equivariant sub-sheaves, which is not true for general stability conditions.

As in the case of Gieseker stability, it is natural to expect that these stability conditions lead to the construction of a moduli spaces of semi-stable sheaves, and thus become a tool to classify sheaves on a given manifold.

\paragraph{Organisation of the paper and main results.} In Section \ref{SEC:Stabilité générale}, we define stability conditions adapted to coherent sheaves of dimension $d$ (Definition \ref{DEF:Adapté}). We then state classical results about these stability conditions such as the existence and uniqueness of Harder--Narasimhan filtrations (Theorem \ref{THE:FHN}) and the existence of Jordan--Hölder filtrations and the uniqueness of the graded object (Theorem \ref{THE:FJH}).

In Section \ref{SEC:Charges centrales polynomiales}, we focus on the particular case of stability conditions coming from a central charge in the large volume limit. We show that being adapted to sheaves of dimension $d$ only depends on the stability vector $\rho$ (Definition \ref{DEF:rho adapté} and Theorem \ref{THE:Caracterisation stabilité}).

We also study the case of the Gieseker stability. We show that it arises from a central charge in the large volume limit and we even construct a central charge $Z$ and an associated operator $\cZ$ such that the asymptotic $Z$-critical equation is Leung's almost Hermitian Einstein equation. This shows that Leung's correspondence \cite{Leung} is a particular case of the Dervan--McCarthy--Sektnan correspondence \cite{DMS} (Example \ref{EX:Leung}).

\paragraph{Acknowledgment.} I thank my PhD advisor, Carl Tipler for his support and feedback and for providing me numerous useful ideas and references. I also thank Ruadha\'\isanspoint\ Dervan, Carlo Scarpa and Lars Martin Sektnan for stimulating conversations about $Z$-critical equations.

\section{General $\mu$-stability}\label{SEC:Stabilité générale}

\subsection{Definitions and basic properties}\label{SEC:Définitions stabilité généralisée}

We study here stability with respect to a large class of stability conditions on the category of coherent sheaves on a smooth projective variety. In all this note, sheaves are assumed to be coherent.

Let $(X,[\omega])$ be a smooth projective variety. Let $(V,\leq)$ be a totally ordered $\Q$-vector space and
$$
\Deg : K(X) \longrightarrow V,
$$
a groups morphism. Such a morphism will play the role of a generalisation of the degree in the context of slope stability and of Hilbert polynomials in the context of Gieseker stability. In the case of Hilbert polynomials, $V = \R[k]$ with $A \leq B$ if $A(k) \leq B(k)$ for $k$ large enough.

When $\cE$ is a non-zero sheaf of dimension $d$ (the dimension of $\cE$ is the dimension of its support), we also introduce
$$
\mu(\cE) = \frac{\Deg(\cE)}{\Rk(\cE)} \in V,
$$
where
$$
\Rk(\cE) = \ch_d(\cE) \cup [\omega]^{n - d} \in \N^*
$$
is the generalised rank of $\cE$. It is always a positive integer. Indeed, let us decompose
$$
\supp(\cE) = \bigcup_{i \in I} V_i \cup W,
$$
where $\dim(W) < d$ and the $V_i$ are the irreducible components of $\supp(\cE)$ of dimension $d$. By assumption, $I \neq \emptyset$ and for all $i \in I$, $\cE_{|V_i}$ has a positive rank. Let $\iota : \supp(\cE) \rightarrow X$ be the inclusion. We have, by the Grotendieck--Riemann--Roch formula,
\begin{align*}
    \ch(\cE) & = \ch(\iota_*\iota^*\cE)\\
    & = \ch(\iota_!\iota^*\cE)\\
    & = \iota_*(\ch(\iota^*\cE)\Td(T_\iota))\\
    & = \iota_*\left(\left(\sum_{i \in I} \rk(\cE_{|V_i})[V_i] + \textrm{ h.d.t.}\right)(1 + \textrm{ h.d.t.})\right)\\
    & = \sum_{i \in I} \rk(\cE_{|V_i})[V_i] + \textrm{ h.d.t.},
\end{align*}
where "h.d.t." stands for "higher degree terms". We deduce that
$$
\Rk(\cE) = \sum_{i \in I} \rk(\cE_{|V_i})\int_{V_i} \omega^d \in \N^*,
$$
because $\omega$ is an integral class.

\begin{definition}\label{DEF:P-stabilité}
    We say that a sheaf $\cE$ is \textit{$\mu$-semi-stable} if for all strict sub-sheaves $0 \subsetneq \cF \subsetneq \cE$,
    \begin{equation}\label{EQ:Stabilité généralisée}
        \mu(\cF) \leq \mu(\cE).
    \end{equation}
    We say that $\cE$ is \textit{$\mu$-stable} if the inequality (\ref{EQ:Stabilité généralisée}) is strong. We say that $\cE$ is \textit{$\mu$-polystable} if it is $\mu$-semi-stable and when (\ref{EQ:Stabilité généralisée}) is an equality, then $\cE = \cF \oplus \cE/\cF$.
\end{definition}

The purpose of this section is to study the case where the $\mu$-stability satisfies the same properties as Gieseker stability (existence and uniqueness of Harder--Narasimhan filtrations, existence of Jordan--Hölder filtrations and uniqueness of the graded object in the semi-stable case, etc.). In the next sub-sections, we study more particularly the case of polynomials given by a central charge in the large volume limit. In this case, we give a simple characterisation of the central charges satisfying this property.

\begin{definition}\label{DEF:Adapté}
    We say that $\mu$ is \textit{adapted to sheaves of dimension $d$} if for all sheaf $\cE$ of pure dimension $d$ and all $\cF \subsetneq \cE$ such that $\cE/\cF$ has dimension less than $d$,
    $$
    \mu(\cF) < \mu(\cE).
    $$
    We say that $\mu$ is \textit{adapted to coherent sheaves} if it is adapted to sheaves of all dimensions. We say that $\mu$ is \textit{adapted to torsion-free sheaves} if it is adapted to sheaves of dimension $n$.
\end{definition}

\begin{example}\label{EX:Exemple mu stabilité adaptée}
    We give here a wide class of examples of stability conditions $\mu$ adapted to coherent sheaves in $V = \R[\epsilon]$ endowed with the order $A \leq B$ if for all $\epsilon > 0$ small enough, $A(\epsilon) \leq B(\epsilon)$. When $\cE$ is a sheaf, we set
    $$
    \ch_{\leq k}(\cE) = \sum_{i = 0}^k \ch_i(\cE).
    $$
    Let $(d_k)_{1 \leq k \leq n}$ be positive integers and $(\Gamma_{k,j})_{1 \leq k \leq n,0 \leq j \leq d_k - 1}$ be a collection of Hodge classes
    $$
    \Gamma_{k,j} \in H^{*,*}(X,\R) = \bigoplus_{i = 0}^n H^{i,i}(X,\C) \cap H^{2i}(X,\R).
    $$
    The stability condition given by
    $$
    \Deg(\cE) = \left(\sum_{k = 1}^n \left(\Gamma_{k,0} \cup \ch_{\leq k}(\cE) + \sum_{j = 1}^{d_k - 1} \Gamma_{k,j} \cup \ch_{\leq k}(\cE)\epsilon^j\right)\epsilon^{d_1 + \cdots + d_{k - 1}}\right)^{(n,n)} \in \R[\epsilon],
    $$
    is adapted to coherent sheaves as long as for all $k$ and all sub-varieties $V$ of $X$ of codimension $k$, $(\Gamma_{0,k} \cup [V])^{(n,n)} > 0$.
\end{example}

\begin{example}
    A particular case of Example \ref{EX:Exemple mu stabilité adaptée} is the stability condition built in \cite{MPT} by Mégy--Pavel--Toma. They consider the map
    $$
    P_\alpha(\cE) = \sum_{i = 0}^n \frac{1}{i!}\left(\ch(\cE) \cup \alpha_i \cup \Td(X)m^i\right)^{(n,n)} \in \R[m].
    $$
    where each $\alpha_i$ is an ample $i$-class. Here, the order on $\R[m]$ is defined with respect to the asymptotic behaviour at infinity. $A \leq B$ if for all $m \geq 0$ large enough, $A(m) \leq B(m)$. We can recover the convention of the asymptotic behaviour at $0^+$ by replacing $m$ by $\epsilon^{-1}$ and by multiplying the expression by a suitable power of $\epsilon$. We could also use the lexicographic order on the coefficients of the polynomials.
\end{example}

\begin{remark}
    If $\cE$ is of pure dimension $d$ and has an irreducible support, then $\Rk(\cE) = \rk(\cE_{|\supp(\cE)})\int_{\supp(\cE)} \omega^d$. In particular, the $\mu$-(semi-)(poly)stability of $\cE$ is independent of the choice of the polarisation $[\omega]$.
    
    If however $\cE$ has a reducible support $\supp(\cE) = V \cup W$, then $\cE = \iota_{V*}\cE_{|V} \oplus \iota_{W*}\cE_{|W}$ where $\iota_V : V \rightarrow X$ and $\iota_W : W \rightarrow X$ are the injections. Thus, $\cE$ can not be $\mu$-stable and its $\mu$-(semi-)(poly)stability may depend of the choice of $[\omega]$.
\end{remark}
\begin{remark}
    Unless $X$ is a curve, the usual slope stability is not adapted to torsion-free sheaves. Indeed, assume that $\Deg(\cE) = \ch_1(\cE) \cup [\omega]^{n - 1}$ is the usual degree with respect to a given polarisation and $\cE$ is torsion-free. For any sub-variety $V \subset X$ of codimension greater than or equal to $2$, set $\cI_V \subset \cO_X$ its ideal sheaf. Then, $\cI_V \otimes \cE \subset \cE$ has the same rank and the same degree as $\cE$.
\end{remark}
\begin{remark}
    The right category to work with for slope stability is $\Coh_{n,n - 1}(X)$, the category of coherent sheaves modulo coherent sheaves of codimension at least $2$. See \cite[Section 1.6]{Huybrechts_Lehn} for reference. In the case where the inequality in Definition \ref{DEF:Adapté} is large instead on strict, we can probably find similarly a good quotient category to work with.
\end{remark}

The following result is the generalisation of a standard result concerning Gieseker stability.
\begin{proposition}
    Assume that $\mu$ is adapted to sheaves of dimension $d$ and let $\cE$ be of pure dimension $d$. Then $\cE$ is $\mu$-(semi-)stable if and only if it is with respect to its saturated sub-sheaves (\textit{i.e.} those whose quotients are of pure dimension $d$).
\end{proposition}
\begin{proof}
The "only if" part is trivial. Assume $\cE$ is $\mu$-(semi-)stable with respect to saturated sub-sheaves and consider a sub-sheaf $0 \subsetneq \cF \subsetneq \cE$ which is not saturated. Let $\cF'$ be its saturation in $\cE$ so $\cF \subset \cF' \subset \cE$ and $\dim(\cF'/\cF) < d$.

Since $\mu$ is adapted to sheaves of dimension $d$, $\mu(\cF) < \mu(\cF') \leq \mu(\cE)$. We deduce that $\cE$ is $\mu$-(semi-)stable.
\end{proof}

\begin{corollary}
    If $\mu$ is adapted to torsion-free sheaves, all rank $1$ torsion-free sheaves (in particular, line bundles) are $\mu$-stable.
\end{corollary}

\subsection{Harder--Narasimhan and Jordan--Hölder filtrations}

For Gieseker stability, we can construct Harder--Narasimhan (HN) filtrations and Jordan--Hölder (JH) filtrations of semi-stable sheaves. If $\mu$ is adapted to sheaves of dimension $d$, we can build HN filtrations of sheaves of pure dimension $d$ and JH filtrations of $\mu$-semi-stable sheaves of pure dimension $d$ similarly. The proofs are the exact same as in \cite[Sections 1.3 and 1.5]{Huybrechts_Lehn} so we don't write them again. We simply state here the analogue of the most important results that leads to the construction of these filtrations.

\begin{lemma}[{\cite[Lemma 1.3.5]{Huybrechts_Lehn}}]\label{LEM:Faisceau déstabilisant max}
    Assume that $\mu$ is adapted to sheaves of dimension $d$ and let $\cE$ be of pure dimension $d$. Then, there is a unique sub-sheaf $0 \subsetneq \cF \subset \cE$ of $\cE$ such that for all sub-sheaves $\cG$ of $\cE$,
    $$
    \mu(\cG) \leq \mu(\cF)
    $$
    and in case of equality, $\cG \subset \cF$. In particular, $\cF$ is $\mu$-semi-stable.
\end{lemma}

\begin{lemma}[{\cite[Proposition 1.2.7]{Huybrechts_Lehn}}]
    If $\cE$ and $\cF$ are $\mu$-semi-stable sheaves of pure dimension $d$ and $\mu(\cE) > \mu(\cF)$, then $\Hom(\cE,\cF) = 0$.
    
    If $\cE$ is $\mu$-stable and $\cF$ is $\mu$-semi-stable both of pure dimension $d$ and $\mu(\cE) = \mu(\cF)$, then any non-zero morphism $\cE \rightarrow \cF$ is an isomorphism.
\end{lemma}

\begin{theorem}[{\cite[Theorem 1.3.4]{Huybrechts_Lehn}}]\label{THE:FHN}
    Assume that $\mu$ is adapted to sheaves of dimension $d$ and let $\cE$ be a sheaf of pure dimension $d$. There is a unique filtration
    $$
    0 = \cE_0 \subsetneq \cE_1 \subsetneq \cdots \subsetneq \cE_l = \cE
    $$
    such that for all $i$, $\cG_i = \cE_i/\cE_{i - 1}$ has pure dimension $d$ and is $\mu$-semi-stable. Moreover, we have,
    $$
    \mu(\cG_1) > \cdots > \mu(\cG_l).
    $$
    We call this filtration the \textit{Harder--Narasimhan filtration} of $\cE$. $\cE_1$ is the sheaf given by Lemma \ref{LEM:Faisceau déstabilisant max}.
\end{theorem}

\begin{theorem}[{\cite[Proposition 1.5.2]{Huybrechts_Lehn}}]\label{THE:FJH}
    Assume that $\mu$ is adapted to sheaves of dimension $d$ and let $\cE$ be a $\mu$-semi-stable sheaf of pure dimension $d$. There is a filtration
    $$
    0 = \cE_0 \subsetneq \cE_1 \subsetneq \cdots \subsetneq \cE_l = \cE
    $$
    such that for all $i$, $\cG_i = \cE_i/\cE_{i - 1}$ has pure dimension $d$, is $\mu$-stable and $\mu(\cG_i) = \mu(\cE)$. Moreover, the sheaf
    $$
    \mathrm{Gr}(\cE) = \bigoplus_{i = 1}^l \cG_i,
    $$
    is $\mu$-polystable and is uniquely determined by $\cE$.
    
    We call such a filtration a \textit{Jordan--Hölder filtration} of $\cE$ and $\mathrm{Gr}(\cE)$ its \textit{graded object}.
\end{theorem}

\section{The case of polynomial central charges}\label{SEC:Charges centrales polynomiales}

In this section, we study the specific case of $\mu$-stability built from central charges as originally introduced by Bayer \cite{Bayer} and for which Dervan--McCarthy--Sektnan found a Kobayashi--Hitchin like correspondence \cite{DMS}. In particular, we show that under suitable conditions on the stability vector $\rho$, they arise from $\mu$-stability conditions adapted to sheaves of dimension $d$.

We use here the convention of the asymptotic development at $0$ instead of the one at infinity.

Let us consider a central charge,
$$
Z(\cE) = \left(\left(\sum_{i = 0}^n \rho_i[\omega]^i\right) \cup U \cup \ch(\cE)\right)^{(n,n)},
$$
Here, the only assumption we make on the $\rho_i$ is that $\Im(\rho_n) > 0$ and $\Im(\overline{\rho_n}\rho_{n - 1}) > 0$. In this case, we call $\rho$ a \textit{stability vector}. We always assume that $U_0 = 1$. We also define their polynomial counterparts, obtained by replacing $\omega$ with $\epsilon^{-1}\omega$ for small positive real numbers $\epsilon$ and normalising by $\epsilon^n$,
$$
Z_\epsilon(\cE) = \left(\left(\sum_{i = 0}^n \rho_i\epsilon^{n - i}[\omega]^i\right) \cup U \cup \ch(\cE)\right)^{(n,n)},
$$

Notice that,
$$
Z_\epsilon(\cE) = \rho_n[\omega^n] + \mathrm{O}(\epsilon).
$$
Therefore, by assumption on $\rho_n$, for all $\epsilon > 0$ small enough, $\Im(Z_\epsilon(\cE)) > 0$. Let $\phi_\epsilon(\cE) \in ]0,1[$ such that $Z_\epsilon(\cE) = \abs{Z_\epsilon(\cE)}\e^{\i\pi\phi_\epsilon(\cE)}$.

\begin{lemma}[{\cite[Lemma 2.8]{DMS}}]\label{LEM:Caractérisation stabilité asymptotique}
    Let $\cF \subset \cE$ be a sub-sheaf. We have equivalence between
    \begin{enumerate}
        \item For $\epsilon > 0$ small enough, $\phi_\epsilon(\cF) <(\leq)\, \phi_\epsilon(\cE)$.
        \item For $\epsilon > 0$ small enough, $\Im(\overline{Z_\epsilon(\cE)}Z_\epsilon(\cF)) <(\leq)\, 0$.
        \item For $\epsilon > 0$ small enough, $-\frac{\mathrm{Re}(Z_\epsilon(\cF))}{\Im(Z_\epsilon(\cF))} <(\leq)\, -\frac{\mathrm{Re}(Z_\epsilon(\cE))}{\Im(Z_\epsilon(\cE))}$.
    \end{enumerate}
\end{lemma}

\begin{definition}
    We say that a sub-sheaf $\cF$ of $\cE$ (strictly) destabilises $\cE$ asymptotically if one of the equivalent conditions of Lemma \ref{LEM:Caractérisation stabilité asymptotique} is fulfilled. We say that $\cE$ is \textit{asymptotically $Z$-(semi-)stable} if it is (strictly) destabilised by no non-zero sub-sheaves.
\end{definition}

It is not trivial that asymptotic $Z$-stability is equivalent to a $\mu$-stability defined in Definition \ref{DEF:P-stabilité}. We prove it in the next sections.

\subsection{Twisted Chern character and degrees}

Let us define the twisted Chern character of an element $a \in K(X)$ as the multi-degree class,
\begin{align*}
    \ch^U(a) & = \ch(a) \cup U\\
    & = \sum_{\underset{i + j \leqslant n}{i,j = 0}}^n \ch_i(a) \cup U_j\\
    & = \sum_{p = 0}^n \sum_{j = 0}^p \ch_{p - j}(a) \cup U_j\\
    & = \sum_{p = 0}^n \left(\ch_p(a) + \sum_{j = 1}^p \ch_{p - j}(a) \cup U_j\right).
\end{align*}
We see that for all $p$, the $(p,p)$ part of $\ch^U(a)$ is
\begin{equation}\label{EQ:Expression chU}
    \ch_p^U(a) = \ch_p(a) + \sum_{j = 1}^p \ch_{p - j}(a) \cup U_j.
\end{equation}
Introducing the twisted Chern class enables us to express in shorter way the central charge of $a \in K(X)$,
$$
Z_\epsilon(a) = \left(\left(\sum_{i = 0}^n \rho_i\epsilon^{n - i}[\omega]^i\right) \cup \ch^U(a)\right)^{(n,n)} = \sum_{i = 0}^n \rho_i\ch_{n - i}^U(a) \cup [\omega]^i\epsilon^{n - i}.
$$

\begin{proposition}\label{PRO:Premiers caractères de Chern}
    Let $\cE$ be a sheaf of codimension $c$. For all $0 \leq i \leq c - 1$, $\ch_i^U(\cE) = 0$ and $\ch_c^U(\cE) = \ch_c(\cE) > 0$.
\end{proposition}
\begin{proof}
The proof is the same as in the beginning of Sub-section \ref{SEC:Définitions stabilité généralisée}. As a consequence of the Grothendieck--Riemann--Roch formula, we have $\ch_i^U(\cE) = 0$ if $i < c$ and
$$
\ch_c^U(\cE) = \ch_c(\cE) = \sum_{V \subset \supp(\cE)} \rk(\cE_{|V})[V]
$$
where the sum ranges on the irreducible components of $\cE$ of codimension $c$.
\end{proof}

\subsection{Adapted stability vectors}

The conditions we require on the stability vector ($\Im(\rho_n) > 0$ and $\Im(\overline{\rho_n}\rho_{n - 1}) > 0$) are quite weak. They are enough to obtain the Kobayashi--Hitchin like correspondence for semi-stable sufficiently smooth vector bundles shown by Dervan--McCarthy--Sektnan \cite[Theorem 1.1]{DMS}. It is due to the fact that the only destabilising sub-sheaves (in the sense of the slope) are locally free so we don't have to bother about singularities. However, on the algebraic point of view, it can lead to pathological behaviour like a torsion-free sheaf destabilised by a sub-sheaf whose quotient is a torsion sheaf.

\begin{example}\label{EX:Contre-exemple conjecture DMS}
    If $X$ has dimension $n \geq 3$, $V \subset X$ is an irreducible sub-variety of codimension $3$ with ideal sheaf $\cI_V \subset \cO_X$ and
    $$
    Z_\epsilon(\cE) = \epsilon^n(-\e^{-\i\epsilon^{-1}[\omega]} \cup \ch(\cE))^{(n,n)}
    $$
    is the central charge of the dHYM equation, then one can compute that,
    $$
    \Im(\overline{Z_\epsilon(\cO_X)Z_\epsilon(i_*\cO_V)}) \underset{\epsilon \rightarrow 0}{\sim} \frac{1}{n!(n - 3)!}\epsilon^3[\omega]^n([\omega]^{n - m} \cup [V]) > 0.
    $$
    Therefore, if asymptotic $Z$-stability for this central charge can be related to a $\mu$-stability condition, this $\mu$-stability condition won't be adapted to torsion-free sheaves. Notice that it builds a counter-example to \cite[Conjecture 1.6]{DMS} even in the large volume regime because the flat metric of $\cO_X$ is a solution to the dHYM equation.
\end{example}

In order to avoid this pathological behaviour and be able to build HN filtrations for example, we need additional assumptions on the $\rho_i$.

\begin{definition}
    When $\rho = (\rho_0,\ldots,\rho_n)$ is a stability vector such that for all $i$, $\Im(\overline{\rho_i}\rho_{i - 1}) > 0$ (in particular, all the $\rho_i$ are non-zero), we say that it is a \textit{Bayer stability vector} \cite[Definition 3.2.1]{Bayer}.
\end{definition}
\begin{definition}\label{DEF:rho adapté}
    Let $0 \leq d \leq n$. When $\rho = (\rho_0,\ldots,\rho_n)$ is a stability vector such that for all $i < d$, $\Im(\overline{\rho_d}\rho_i) > 0$, we say that $\rho$ is \textit{adapted to sheaves of dimension $d$}. If $d = n$, we say that it is \textit{adapted to torsion-free sheaves}. If it holds for all $d$, we say that $\rho$ is \textit{adapted to coherent sheaves}.
\end{definition}

Of course, stability vectors adapted to coherent sheaves are Bayer and adapted to torsion-free sheaves. The reciprocal is true.

\begin{lemma}\label{LEM:Bayer adapté}
    A stability vector $\rho$ is adapted to coherent sheaves if and only if it is Bayer and adapted to torsion-free sheaves.
\end{lemma}
\begin{proof}
Assume $\rho$ is Bayer and adapted to torsion-free sheaves. Let $0 \leq i < j \leq n - 1$. $\Im(\overline{\rho_n}\rho_j) > 0$ so $(\rho_n,\rho_j)$ is an $\R$-base of $\C$. Let us write $\rho_{j + 1} = a\rho_j + b\rho_n$. $\Im(\overline{\rho_n}\rho_{j + 1}) > 0$ so $a\Im(\overline{\rho_n}\rho_j) > 0$. Since we also have $\Im(\overline{\rho_n}\rho_j) > 0$, $a > 0$. Similarly, $b > 0$. Therefore,
$$
\Im(\overline{\rho_{j + 1}}\rho_i) = a\Im(\overline{\rho_j}\rho_i) + b\Im(\overline{\rho_n}\rho_i) > a\Im(\overline{\rho_j}\rho_i).
$$
By induction on $j \geq i + 1$, $\Im(\overline{\rho_j}\rho_i) > 0$ for all $i < j$ so $\rho$ is adapted to coherent sheaves.
\end{proof}

\subsection{Characterisation of asymptotic $Z$-stability when the stability vector is adapted}

\begin{definition}
    We define the \textit{generalised degrees} and \textit{generalised slopes} of a sheaf $\cE$ of pure codimension $c$ as follows.
    $$
    \deg_i^{U,[\omega]}(\cE) = \ch_i^U(a) \cup [\omega]^{n - i}, \qquad \mu_i^{U,[\omega]}(\cE) =
    \left\{
    \begin{array}{rl}
        +\infty & \textrm{if } i < c,\\
        \displaystyle \frac{\deg_i^{U,[\omega]}(\cE)}{\Rk(\cE)} & \textrm{else}.
    \end{array}
    \right.
    $$
    We also define the \textit{slopes vector} of $\cE$ as $\mu^{U,[\omega]}(\cE) = (\mu_i^{U,[\omega]}(\cE))_{0 \leq i \leq n} \in (\R \cup \{+\infty\})^{n + 1}$. Notice that by Proposition \ref{PRO:Premiers caractères de Chern},
    $$
    \deg_c^U(\cE) = \Rk(\cE).
    $$
    In particular, slopes are well-defined and $\mu^{U,[\omega]}(\cE) = (+\infty,\ldots,+\infty,1,*,\ldots,*)$ where the $1$ is at position $c$.
\end{definition}
We can rewrite the central charge in function of these degrees,
\begin{equation}\label{EQ:Expression charge centrale}
    Z_\epsilon(\cE) = \sum_{i = 0}^n \rho_{n - i}\deg_i^{U,[\omega]}(\cE)\epsilon^i.
\end{equation}
Let $\lex$ be the lexicographic order of $(\R \cup \{+\infty\})^{n + 1}$ where the first coefficient is the one with highest weight. It is a total order.

\begin{theorem}\label{THE:Caracterisation stabilité}
    Let $\cE$ be a sheaf of pure dimension $d$ and $\cF$ a sub-sheaf. Assume that $\rho$ is adapted to sheaves of dimension $d$. Then,
    \begin{enumerate}
        \item $\cF$ destabilises $\cE$ strictly if and only if $\mu^{U,[\omega]}(\cF) \gtx \mu^{U,[\omega]}(\cE)$,
        \item $\cF$ destabilises $\cE$ if and only if $\mu^{U,[\omega]}(\cF) \gex \mu^{U,[\omega]}(\cE)$,
        \item $\cF$ doesn't destabilise $\cE$ if and only if $\mu^{U,[\omega]}(\cF) \ltx \mu^{U,[\omega]}(\cE)$.
    \end{enumerate}
\end{theorem}
\begin{proof}
We have $\dim(\cF) \leq \dim(\cE) = d$. Let us start with the case where $\dim(\cF) < \dim(\cE)$. In this case, $\mu^{U,[\omega]}(\cF) \gtx \mu^{U,[\omega]}(\cE)$ thus we must show that $\cF$ destabilises $\cE$ strictly. Indeed, by Proposition \ref{PRO:Premiers caractères de Chern} and (\ref{EQ:Expression charge centrale}),
$$
Z_\epsilon(\cE) = \sum_{i = 0}^n \rho_{n - i}\deg_i^{U,[\omega]}(\cE)\epsilon^i = \rho_{\dim(\cE)}\deg_{\codim(\cE)}^{U,[\omega]}(\cE)\epsilon^{\codim(\cE)} + \mathrm{O}(\epsilon^{\codim(\cE) + 1}),
$$
and a similar formula holds for $\cF$. Therefore,
\begin{align*}
    \Im(\overline{Z_\epsilon(E)}Z_\epsilon(\cF)) & = \Im(\overline{\rho_{\dim(\cE)}}\rho_{\dim(\cF)})\deg_{\codim(\cE)}^{U,[\omega]}(\cE)\deg_{\codim(\cF)}^{U,[\omega]}(\cF)\epsilon^{\codim(\cE) + \codim(\cF)}\\
    & + \mathrm{O}(\epsilon^{\codim(\cE) + \codim(\cF) + 1}).
\end{align*}
Using the fact that the bottom degrees $\deg_{\codim(\cE)}^{U,[\omega]}(\cE) = \Rk(\cE)$ and $\deg_{\codim(\cF)}^{U,[\omega]}(\cF) = \Rk(\cF)$ are positive and that $\Im(\overline{\rho_{\dim(\cE)}}\rho_{\dim(\cF)}) > 0$ by assumption on the $\rho_i$ and the fact that $\dim(\cE) > \dim(\cF)$, we deduce that this quantity is positive for $\epsilon$ small thus $\cF$ strictly destabilises $\cE$.

Assume now that $\dim(\cF) = \dim(\cE)$ and set $c = \codim(\cF) = \codim(\cE)$. Recall that $\Im(\overline{\rho_{n - c}}\rho_i) > 0$ when $i < c$. For simplicity, we use the convention that $\rho_i = \deg_i^{U,[\omega]}(\cE) = \deg_i^{U,[\omega]}(\cF) = 0$ if $i < 0$ or $i > n$.
\begin{align*}
    \Im(\overline{Z_\epsilon(\cE)}Z_\epsilon(\cF)) & = \Im\left(\left(\sum_{i = 0}^n \overline{\rho_{n - i}}\deg_i^{U,[\omega]}(\cE)\epsilon^i\right)\left(\sum_{j = 0}^n \rho_{n - j}\deg_j^{U,[\omega]}(\cF)\epsilon^j\right)\right)\\
    & = \sum_{p = 0}^{2n}\sum_{j \in \Z} \Im(\overline{\rho_{n - p + j}}\rho_{n - j})\underbrace{\deg_{p - j}^{U,[\omega]}(\cE)}_{= 0 \textrm{ if } j > p - c}\ \underbrace{\deg_j^{U,[\omega]}(\cF)}_{= 0 \textrm{ if } j < c}\epsilon^p\\
    & = \sum_{p = 2c}^{2n}\sum_{j = c}^{p - c} \Im(\overline{\rho_{n - p + j}}\rho_{n - j})\deg_{p - j}^{U,[\omega]}(\cE)\deg_j^{U,[\omega]}(\cF)\epsilon^p.\\
\end{align*}
At $p = 2c$, the only value that takes $j$ is $c$ thus $\Im(\overline{\rho_{n - p + j}}\rho_{n - j}) = \Im(\overline{\rho_{n - c}}\rho_{n - c}) = 0$. We deduce that the coefficient at $\epsilon^{2c}$ vanishes. More generally, if $p$ is even and $j = \frac{p}{2}$, $\Im(\overline{\rho_{n - p + j}}\rho_{n - j}) = 0$. When $p > c$, $j$ takes at least two values and the coefficient $a_p$ at order $p$ of the polynomial $\epsilon \mapsto \Im(\overline{Z_\epsilon(\cE)}Z_\epsilon(\cF))$ is,
\begin{align}
    a_p & = \sum_{j = c}^{p - c} \Im(\overline{\rho_{n - p + j}}\rho_{n - j})\deg_{p - j}^{U,[\omega]}(\cE)\deg_j^{U,[\omega]}(\cF)\nonumber\\
    & = \sum_{j = c}^{\left\lfloor \frac{p}{2} \right\rfloor} \Im(\overline{\rho_{n - p + j}}\rho_{n - j})\deg_{p - j}^{U,[\omega]}(\cE)\deg_j^{U,[\omega]}(\cF) + \sum_{j = \left\lceil \frac{p}{2} \right\rceil}^{p - c} \Im(\overline{\rho_{n - p + j}}\rho_{n - j})\deg_{p - j}^{U,[\omega]}(\cE)\deg_j^{U,[\omega]}(\cF)\nonumber\\
    & = -\sum_{j = c}^{\left\lfloor \frac{p}{2} \right\rfloor} \Im(\overline{\rho_{n - j}}\rho_{n - p + j})\deg_{p - j}^{U,[\omega]}(\cE)\deg_j^{U,[\omega]}(\cF) + \sum_{j = c}^{\left\lfloor \frac{p}{2} \right\rfloor} \Im(\overline{\rho_{n - j}}\rho_{n - p + j})\deg_j^{U,[\omega]}(\cE)\deg_{p - j}^{U,[\omega]}(\cF)\nonumber\\
    & = \sum_{j = c}^{\left\lfloor \frac{p}{2} \right\rfloor} \Im(\overline{\rho_{n - j}}\rho_{n - p + j})(\deg_j^{U,[\omega]}(\cE)\deg_{p - j}^{U,[\omega]}(\cF) - \deg_{p - j}^{U,[\omega]}(\cE)\deg_j^{U,[\omega]}(\cF))\nonumber\\
    & = \sum_{j = c}^{\left\lfloor \frac{p}{2} \right\rfloor} \Im(\overline{\rho_{n - j}}\rho_{n - p + j})\deg_c^{U,[\omega]}(\cE)\deg_c^{U,[\omega]}(\cF)(\mu_j^{U,[\omega]}(\cE)\mu_{p - j}^{U,[\omega]}(\cF) - \mu_{p - j}^{U,[\omega]}(\cE)\mu_j^{U,[\omega]}(\cF)).\label{EQ:Expression ap}
\end{align}
Let $c \leq c' \leq n$ the largest integer such that $\mu_{c'}^{U,[\omega]}(\cF) = \mu_{c'}^{U,[\omega]}(\cE)$. It is well-defined because $\mu_c^{U,[\omega]}(\cF) = \mu_c^{U,[\omega]}(\cE) = 1$. If $c' = n$, then $\mu^{U,[\omega]}(\cF) = \mu^{U,[\omega]}(\cE)$ and by Equation (\ref{EQ:Expression ap}), $\Im(\overline{Z_\epsilon(\cE)}Z_\epsilon(\cF)) = 0$ so $\cF$ destabilises $\cE$ but not strictly.

Assume now that $c' < n$. Using Equation (\ref{EQ:Expression ap}), when $p \leq c' + c$, all terms of the form $\mu_j^{U,[\omega]}(\cE)\mu_{p - j}^{U,[\omega]}(\cF) - \mu_{p - j}^{U,[\omega]}(\cE)\mu_j^{U,[\omega]}(\cF)$ vanish when $c \leq j \leq \left\lfloor \frac{p}{2} \right\rfloor$. When $p = c' + c + 1$, all such term vanish too expect when $j = c = p - c' - 1$. We deduce that,
$$
\forall p \leq c' + c, a_p = 0,
$$
and,
\begin{align*}
    a_{c' + c + 1} & = \Im(\overline{\rho_{n - c}}\rho_{n - c' - 1})\deg_c^{U,[\omega]}(\cE)\deg_c^{U,[\omega]}(\cF)(\mu_c^{U,[\omega]}(\cE)\mu_{c' + 1}^{U,[\omega]}(\cF) - \mu_{c' + 1}^{U,[\omega]}(\cE)\mu_c^{U,[\omega]}(\cF))\\
    & = \Im(\overline{\rho_{n - c}}\rho_{n - c' - 1})\deg_c^{U,[\omega]}(\cE)\deg_c^{U,[\omega]}(\cF)(\mu_{c' + 1}^{U,[\omega]}(\cF) - \mu_{c' + 1}^{U,[\omega]}(\cE)).
\end{align*}
We obtain the wanted equivalence because $\Im(\overline{\rho_{n - c}}\rho_{n - c' - 1})\deg_c^{U,[\omega]}(\cE)\deg_c^{U,[\omega]}(\cF) > 0$.
\end{proof}

\begin{corollary}
    If $\rho$ is adapted to sheaves of dimension $d$ and $\cE$ has dimension $d$, it is asymptotically $Z$-(semi-)stable if and only if for all $0 \subsetneq \cF \subsetneq \cE$, $\mu^{U,[\omega]}(\cF) \ltx (\lex)\ \mu^{U,[\omega]}(\cE)$.
\end{corollary}

\begin{corollary}
    If $\rho$ is adapted to sheaves of dimension $d$, if $\cE$ doesn't have pure dimension, then it admits a sub-sheaf of strictly lower dimension, which destabilises it strictly.
\end{corollary}

\begin{corollary}
    Once we assume that $\rho$ is adapted to sheaves of dimension $d$, stability of sheaves of pure dimension $d$ doesn't depend on the particular choice of $\rho$.
\end{corollary}

For each $d$, we define the stability condition
$$
P_{Z,d} : \cE \mapsto (\deg_{n - d}^{U,[\omega]}(\cE),\ldots,\deg_n^{U,[\omega]}(\cE)) \in \R^d.
$$
$\R^d$ is endowed with its lexicographic order. Similarly, we could have used the space of real polynomials. A consequence of Theorem \ref{THE:Caracterisation stabilité} is that, if $\rho$ is adapted to sheaves of dimension $d$, for any sheaf $\cE$ of pure dimension $d$, $\cE$ is asymptotically $Z$-(semi-)(poly)stable if and only if it is $P_{Z,d}$-(semi-)(poly)stable. Moreover, $P_{Z,d}$ is itself adapted to sheaves of dimension $d$ in the sense of Definition \ref{DEF:Adapté}.

\subsection{An example : Gieseker stability}

We show here that Gieseker stability is a particular case of asymptotic $Z$-stability. In \cite[Section 4.2]{Keller_Scarpa}, Gieseker stability is realised as an asymptotic $Z$-stability but with a stability vector $\rho$ depending on the parameter $k = \epsilon^{-1}$. We show here that we can choose $\rho$ independent of this parameter.

Let $U = \Td(X)$, $\rho$ any stability vector adapted to torsion-free sheaves and $\omega \in c_1(\cL)$ be any Kähler form with $\cL$ an ample bundle. Then, for all torsion-free sheaf $\cE$ and all integer $k$, by the Hirzebruch--Riemann--Roch formula,
\begin{align*}
    \chi(\cE \otimes \cL^k) & = \int_X \Td(X) \cup \ch(E \otimes \cL^k)\\
    & = \int_X \ch^{\Td(X)}(\cE) \cup \e^{k[\omega]}\\
    & = \sum_{i = 0}^n \ch_i^{\Td(X)}(\cE) \cup \frac{[\omega]^i}{i!}k^i\\
    & = \sum_{i = 0}^n \frac{\deg_i^{\Td(X),[\omega]}(\cE)}{i!}k^i.
\end{align*}

By Theorem \ref{THE:Caracterisation stabilité}, we easily see that $\cE$ is Gieseker stable if and only if it is asymptotically $Z$-stable for the central charge $Z$ associated with $([\omega],\rho,\Td(X))$.

\begin{remark}
By choosing $\rho$ adapted to coherent sheaves, we obtain a natural way to extend Gieseker stability to any coherent sheaf. Gieseker stability so extended to all sheaves is sometimes referred to as Simpson stability \cite{Simpson2,Rudakov}. In \cite[2.1]{Bayer}, Bayer notices that this stability condition can be expressed as a polynomial central charge. The above proposition shows that it can actually be expressed as the particular kind of central charge given by \cite[Theorem 3.2.2]{Bayer} by choosing the stability vector adapted to coherent sheaves.
\end{remark}

We show in the next example that with the right choice of $\rho$, we can recover Leung's almost Hermitian Einstein equation \cite{Leung}.

\begin{example}\label{EX:Leung}
    If we set $\rho_n = \i$, for all $k < n$, $\rho_k = \frac{1}{k!}(\i - 1)$ and $u = \td(X,\omega)$ the canonical representative of the Todd class of $X$, for all holomorphic Hermitian vector bundle $(\cE,h)$ the $Z_\epsilon$-critical equation is Leung's almost Hermitian-Einstein equation \cite{Leung}.
\end{example}
\begin{proof}
Recall that $\hat{F}(h)$ is the reduced curvature form of $(\cE,h) = (E,\dbar,h)$. We have,
\begin{align*}
    \Im(\overline{Z_\epsilon(\cE)}\cZ_\epsilon(\cE,h)) & = \Im\left(\left(\sum_{i = 0}^n \overline{\rho_{n - i}}\deg_i^{\Td(X),[\omega]}(\cE)\epsilon^i\right)\left(\left(\sum_{j = 0}^n \rho_{n - j}\epsilon^j\omega^{n - j}\right) \wedge \td(X,\omega) \wedge \e^{\hat{F}(h)}\right)\right)^{(n,n)}\\
    & = \left(\sum_{i,j = 0}^n \underbrace{\Im(\overline{\rho_{n - i}}\rho_{n - j})}_{= 0 \textrm{ if } ij \neq 0}\deg_i^{\Td(X),[\omega]}(\cE)\omega^{n - j} \wedge \td(X,\omega) \wedge \e^{\hat{F}(h)}\epsilon^{i + j}\right)^{(n,n)}\\
    & = \Biggl(\sum_{j = 1}^n \frac{1}{(n - j)!}\rk(\cE)\omega^{n - j} \wedge \td(X,\omega) \wedge \e^{\hat{F}(h)}\epsilon^j\\
    &\ \ \ \  - \sum_{i = 1}^n \frac{1}{(n - i)!}\deg_i^{\Td(X),[\omega]}(\cE)\omega^n \wedge \td(X,\omega) \wedge \e^{\hat{F}(h)}\epsilon^i\Biggl)^{(n,n)}\\
    & = \rk(\cE)\epsilon^n\left(\e^{\hat{F}(h) + \epsilon^{-1}\omega\Id_E} \wedge \td(X,\omega)\right)^{(n,n)} - \sum_{i = 0}^n \frac{1}{(n - i)!}\deg_i^{\Td(X),[\omega]}(\cE)\omega^n\Id_E.
\end{align*}
Therefore, $h$ is solution if and only if
$$
\left(\e^{\hat{F}(h) + \epsilon^{-1}\omega\Id_E} \wedge \td(X,\omega)\right)^{(n,n)} = \frac{1}{\Vol(X)}p_\cE(\epsilon^{-1})\frac{\omega^n}{n!}\Id_E.
$$
This is exactly Leung's almost Hermitian Einstein equation \cite{Leung}.
\end{proof}

Notice that the stability vector of Example \ref{EX:Leung} is adapted to torsion-free sheaves but not to coherent sheaves. It motivates the correspondence between asymptotic $Z$-stability and the existence $Z_\epsilon$-critical connections for $\epsilon > 0$ small. In particular, the Leung correspondence \cite[Theorem 1]{Leung} is a particular case of the Dervan--McCarthy--Sektnan correspondence \cite[Theorem 1.1]{DMS}.

\subsection{A characterisation of being adapted}

Recall that a stability vector is adapted to coherent sheaves if and only if it is adapted to torsion-free sheaves and Bayer (Lemma \ref{LEM:Bayer adapté}). In this sub-section, we give some other necessary and sufficient conditions on the polynomial central charge for the stability vector to be adapted to coherent sheaves. In particular, point (3) justifies the terms "adapted to coherent sheaves".

\begin{proposition}\label{PRO:Caractérisation adapté}
    Let $\rho$ be a Bayer stability vector and $U = 1 + \sum_{i = 1}^n U_i \in H^{*,*}(X,\C)$. Let $Z_\epsilon$ be the associated polynomial central charge. We have equivalence between,
    \begin{enumerate}
        \item $\rho$ is adapted to torsion-free sheaves.
        \item $\rho$ is adapted to coherent sheaves.
        \item Up to a common rotation of the $\rho_i$, the heart of bounded $t$-structure $\cA$ built by Bezrukavnikov and Kashiwara \cite{Bezrukavnikov}\cite{Kashiwara}\cite[Theorem 3.1.2]{Bayer} associated to the Bayer stability condition $Z_\epsilon$ is $\Coh(X)$ up to an even number of shifts.
        \item $\Im(\overline{\rho_n}\rho_0) > 0$ and there is a $\lambda \in \C^*$ such that for all non-zero sheaf $\cS$, for all $\epsilon$ small enough (depending on $\cS$), $\lambda Z_\epsilon(\cS) \in \H$ where,
        $$
        \H = \{z \in \C|\Im(z) > 0 \textrm{ or } z \in \R_-^*\} = \{r\e^{i\pi\theta}|r > 0, 0 < \theta \leq 1\}.
        $$
        \item $\Im(\overline{\rho_n}\rho_0) \neq 0$ and there is a $\lambda \in \C^*$ such that for all smooth sub-variety $i : V \rightarrow X$, for all $\epsilon$ small enough (depending on $V$), $\lambda Z_\epsilon(i_*\cO_V) \in \H$.
    \end{enumerate}
\end{proposition}
\begin{proof}\ \\

\noindent\framebox{$1 \Leftrightarrow 2$} Lemma \ref{LEM:Bayer adapté}.\\

\noindent\framebox{$2 \Rightarrow 3$} Up to rotating the $\rho_i$, we may assume that $\rho_0 \in \R_-^*$ so for all $i \geq 1$, $\Im(\rho_i) > 0$. Using the notations of \cite{Bayer}, a perversity function $p$ associated to $\rho$ \cite[Definitions 3.1.1, 3.2.1]{Bayer} is necessarily constant and even. We easily compute that the associated heart of bounded $t$-structure is $\cA = \Coh(X)[p]$ \cite[Theorem 3.1.2]{Bayer}.\\

\noindent\framebox{$3 \Rightarrow 4$} By \cite[Definition 2.3.2, Theorem 3.2.2]{Bayer}, there is a continuous germ $\phi_0 : (\R \cup \{+\infty\},+\infty) \rightarrow \R$ such that for all sheaf $\cS$ and all $\epsilon > 0$ small enough,
$$
\phi_0(\epsilon^{-1}) < \phi(\cS)(\epsilon^{-1}) \leq \phi_0(\epsilon^{-1}) + 1,
$$
where $\phi(\cS)(\epsilon^{-1})$ is a well chosen argument of $Z_\epsilon(\cS)$ divided by $\pi$. The proof of \cite[Theorem 3.2.2]{Bayer} shows that we can choose $\phi_0$ constant hence the result with $\lambda = \e^{-i\pi\phi_0}$. We only have left to show that $\Im(\overline{\rho_n}\rho_0) > 0$.

The assumption on the heart $\cA$ implies that the perversity function associated to $\rho$ (up to a rotation) is constant and even thus all $\rho_i$ can be written as $\abs{\rho_i}\e^{i\pi\theta_i}$ with $0 < \theta_i \leq 1$. The assumption that $\rho$ is Bayer implies that for all $i$, $\theta_i < \theta_{i - 1}$. Therefore, $\theta_n < \theta_0$ so $\Im(\overline{\rho_n}\rho_0) > 0$.\\

\noindent\framebox{$4 \Rightarrow 5$} Trivial with the same $\lambda$.\\

\noindent\framebox{$5 \Rightarrow 2$} Let $V$ be an intersection of $p \geq 1$ generic hyperplanes of $\P^N$ with $X$ in $X \hookrightarrow \P^N$. By Bertini's theorem, $V$ is smooth and has pure codimension $p$ in $X$. By Proposition \ref{PRO:Premiers caractères de Chern}, $\deg_k^{U,[\omega]}(i_*\cO_V) = 0$ if $k < p$ and $\deg_p^{U,[\omega]}(i_*\cO_V) > 0$. Therefore,
$$
\lambda Z_\epsilon(i_*\cO_V) = \lambda\sum_{i = 0}^n \rho_{n - i}\deg_i^{U,[\omega]}(i_*\cO_V)\epsilon^i = \lambda\rho_{n - p}\deg_p^{U,[\omega]}(i_*\cO_V)\epsilon^p + \mathrm{O}(\epsilon^{p + 1}).
$$
By assumption, this quantity belongs to $\H$ for $\epsilon$ small enough so $\lambda\rho_{n - p} \in \overline{\H}$ for all $p$. We deduce that for all $i$, $\rho_i = \abs{\rho_i}\overline{\lambda}\e^{\i\pi\theta_i}$ for some $0 \leq \theta_i \leq 1$. The inequality $\Im(\overline{\rho_i}\rho_{i - 1}) > 0$ which holds by definition of Bayer stability vectors implies that $\abs{\rho_i} > 0$ and $\theta_i < \theta_{i - 1}$. Thus, when $i < j$, $\theta_j < \theta_i$ meaning that $\Im(\overline{\rho_j}\rho_i) > 0$ except maybe if $\theta_j = 0$ and $\theta_i = 1$.

For all $1 \leq i \leq n$, $0 \leq \theta_i < \theta_{i - 1} \leq 1$ so this equality can only occur at $j = n$ and $i = 0$ which is impossible by assumption.
\end{proof}

\begin{remark}
    If $P = P_{Z,d}$ and $\rho$ satisfies the equivalent conditions of Proposition \ref{PRO:Caractérisation adapté}, then the existence of HN filtrations for sheaves of pure dimension $d$ given by Theorem \ref{THE:FHN} can be recovered by Bayer's construction \cite[Theorem 3.2.2]{Bayer}.
\end{remark}

\bibliographystyle{plainurl}

\end{document}